\documentclass[11pt, a4paper]{amsart}
\numberwithin{equation}{section}
\usepackage{luatex85}
\usepackage{amscd,amsthm,amsfonts,amssymb,amsmath,euscript,mathrsfs}
\usepackage[dvips]{graphicx}
\usepackage[dvips]{graphics}
\usepackage[matrix,arrow]{xy}
\usepackage{epsfig}
\usepackage{setspace}
\usepackage{longtable}
\usepackage{url}
\usepackage{comment}

\usepackage[T1]{fontenc}
\usepackage[utf8]{inputenc}

\newcounter{statements}

\newcounter{conjectures}

\newtheorem{theorem}[statements]{Theorem}

\newtheorem{lemma}[statements]{Lemma}

\newtheorem{corollary}[statements]{Corollary}
\newtheorem{proposition}[statements]{Proposition}

\theoremstyle{definition}
\newtheorem{example}[statements]{Example}
\newtheorem{definition}[statements]{Definition}

\theoremstyle{remark}
\newtheorem{remark}[statements]{Remark}

\newcommand{\QQ}{{\mathbb Q}}
\newcommand{\ZZ}{{\mathbb Z}}

\newcommand{\bbA}{{\mathbb A}}
\newcommand{\bbC}{{\mathbb C}}

\newcommand{\bbZ}{{\mathbb Z}}

\newcommand{\bbQ}{{\mathbb Q}}

\newcommand{\Hom}{\mathrm{Hom}\,}

\newcommand{\cO}{\mathcal{O}}

\newcommand{\Ext}{\mathrm{Ext}}
\newcommand{\Sym}{\mathrm{Sym}}
\newcommand{\cH}{\mathcal{H}}
\newcommand{\cM}{\mathcal{M}}
\newcommand{\Hscr}{\mathcal{H}}
\newcommand{\cR}{\mathcal{R}}

\newcommand{\Rscr}{\mathcal{R}}
\newcommand{\Mscr}{\mathcal{M}}
\def\Gr{\operatorname{Gr}}
\def\quot{/\!\!/}

\author[Valery Lunts, \v{S}pela \v{S}penko and Michel Van den Bergh]{Valery Lunts, \v{S}pela \v{S}penko and Michel
  Van den Bergh} 
\address[Valery Lunts]{Department of Mathematics, Indiana University Bloomington, Rawles Hall 251, 831 East 3rd St., Bloomington, IN 47405-7106}
\address{National Research University Higher School of Economics, Moscow}
\email{vlunts@indiana.edu}
\address[\v{S}pela \v{S}penko]{D\'epartement de Math\'ematique, Universit\'e Libre de Bruxelles, Campus de la Plaine CP 213, Bld du Triomphe, B-1050 Bruxelles}
\email{spela.spenko@ulb.be}
\address[Michel Van den Bergh]{Vakgroep Wiskunde, Universiteit Hasselt, Universitaire Campus \\
  B-3590 Diepenbeek}
\email{michel.vandenbergh@uhasselt.be}
\address{Vrije Universiteit Brussel, Department of Mathematics and Data Science, Pleinlaan 2\\B-1050 Brussel} 
\email{michel.van.den.bergh@vub.be}
\thanks{ The first author was supported by the Basic Research Program of the National
Research University Higher School of Economics. The second author is supported by a MIS grant from the National Fund for Scientific Research (FNRS) and an ACR grant from the Université Libre de Bruxelles. The third author is a senior researcher at the Research
  Foundation Flanders (FWO).  While working on this project he was
  supported by the ERC grant SCHEMES and the FWO grant G0D8616N: ``Hochschild cohomology and
  deformation theory of triangulated categories''.}

\newdimen\uboxsep \uboxsep=1ex
\def\uboxn#1{\vtop to 0pt{\hrule height 0pt depth 0pt\vskip\uboxsep
\hbox to 0pt{\hss #1\hss}\vss}}

\def\uboxs#1{\vbox to 0pt{\vss\hbox to 0pt{\hss #1\hss}
\vskip\uboxsep\hrule height 0pt depth 0pt}}

\let\oldmarginpar\marginpar
\def\marginpar#1{\oldmarginpar{\raggedright \tiny \offinterlineskip #1}}

\def\CoHA{\operatorname{CoHA}}
\def\Rep{\operatorname{Rep}}
\def\GL{\operatorname{GL}}
\def\KHA{\operatorname{KHA}}
\def\Mod{\operatorname{Mod}}
\def\Aut{\operatorname{Aut}}
\mathchardef\-="2D
  
\begin{document}

\title[On cohomological and K-theoretical Hall algebras]{On cohomological and K-theoretical Hall algebras of symmetric quivers}
%

\begin{abstract}
We give a brief review of the cohomological Hall algebra CoHA $\Hscr$ and the K-theoretical Hall algebra $\KHA$ $\cR$ associated to quivers. 
In the case of symmetric quivers, we show that there exists a homomorphism of algebras (obtained from a Chern character map) $\cR\to \hat{\Hscr}^{\tilde\sigma}$ where $\hat{\Hscr}^{\tilde\sigma}$ is a Zhang twist of the completion of $\Hscr$. Moreover, we establish the equivalence  of  categories of ``locally finite'' graded modules $\hat{\cH}^{\tilde\sigma}\-\Mod_{lf}\simeq \cR _\bbQ \-\Mod_{lf}$. Examples of locally finite $\hat{\cH}^{\tilde\sigma}$-, resp. $\cR _\bbQ$- modules appear naturally as the cohomology, resp. K-theory, of framed moduli spaces of quivers.
\end{abstract}
\maketitle

\section{Introduction}
Let $Q$ be a finite oriented quiver with the set of vertices $I$ and with $a_{ij}$ arrows from $i\in I$ to $j\in I$, so that $a_{ij}\in \bbZ _{\geq 0}$. Fix a dimension vector $\gamma =(\gamma ^i)_{i\in I}\in \bbZ _{\geq 0}^I$. 
Let $M_\gamma$ be the affine space of representations of $Q$ of dimension vector $\gamma$. 
The variety $M_\gamma$ carries a natural action of 
 the group
$G_\gamma :=\prod _{i\in I}\GL(\gamma ^i,\bbC)$.

Let $\Hscr=\oplus_\gamma H_{G_\gamma}(M_\gamma,\QQ)$. There is a natural product on $\Hscr$ coming from extension of representations. The corresponding $\ZZ^I_{\geq 0}$-graded algebra is called {\em the cohomological Hall algebra} (CoHA), it was introduced by Kontsevich and Soibelman \cite{KS}. 

Similarly, the $\ZZ^I_{\geq 0}$-graded algebra $\Rscr=\oplus_\gamma K^0_{G_\gamma}(M_\gamma)$ with multiplication defined in an analogous way    is called {\em the K-theoretical Hall algebra} ($\KHA$), it was introduced by P{$\rm\breve{a}$}durariu \cite{Pa,Tudor1}. 

One can construct a natural (Chern character) homomorphism $\Rscr$ to the completion  $\hat\Hscr$ of $\Hscr$ which is an inclusion of $\ZZ^I_{\geq 0}$-graded spaces. However, in order to  obtain a multiplicative map (in the case of a symmetric quiver) we need to twist the product on $\hat{\Hscr}$ with an appropriate Zhang twist. 

\begin{theorem}\label{thm}(Theorem \ref{main thm})
Assume that $Q$ is symmetric. Then there exists a group homomorphism $\sigma:\ZZ^I\to \Aut(\hat{\Hscr})$ such that (a twist of) the Chern character map  $\Rscr\to \hat{\Hscr}^{\tilde\sigma}$ where $\hat{\Hscr}^{\tilde\sigma}$ is the Zhang twist of $\hat{\Hscr}$ is an injective homomorphism of $\ZZ^I_{\geq 0}$-graded algebras.
\end{theorem}

\begin{remark}
The twist of the Chern character map descends to associated graded algebras where it coincides with the usual Chern character map. In this case the homomorphism property was already established in \cite{Pa}.
\end{remark}

Moreover, we show that certain module categories of $\hat{\Hscr}$ and $\Rscr$ are equivalent. 
By $\Hscr\-\Mod_{lf}$, resp. $\Rscr_\QQ\-\Mod_{lf}$, we denote the abelian category of ``locally finite'' graded $\Hscr$-modules, resp. $\Rscr_\QQ$-modules (cf. Definition \ref{def of large ideal and loc f module}).


\begin{proposition}(Corollary \ref{summary of equivalences})
For a symmetric quiver the following categories of locally finite graded modules are equivalent 
\begin{equation}
  \label{eq:equivalence}
\cH\-\Mod_{lf}\simeq \hat{\cH}\-\Mod_{lf}\simeq
\hat{\cH}^{\tilde{\sigma}}\-\Mod_{lf}\simeq \cR _\bbQ\-\Mod_{lf}
.
\end{equation}
\end{proposition}

Locally finite graded $\hat{\cH}$, resp. $\Rscr_\QQ$-modules appear naturally as the cohomology, resp. K-theory of the framed moduli space of $Q$. 

\begin{proposition}(Propositions \ref{prop that geom mod is locally finite}, \ref{prop:Rst_locally_finite}) 
Let $\Mscr_\gamma$ be the moduli space of representations of the augmented quiver $\tilde{Q}$ of dimension vector $\tilde{\gamma}=(1,\gamma)$. Then $\oplus_\gamma H(\Mscr_\gamma,\QQ)\in\hat{\cH}\-\Mod_{lf}$, $\oplus_\gamma K^0(\Mscr_\gamma)_\QQ\in \Rscr_\QQ\-\Mod_{lf}$. If $Q$ is symmetric then these two modules correspond to each other under the equivalence \eqref{eq:equivalence}, via the (twisted) Chern character (from Theorem \ref{thm}).
\end{proposition}

\section{Review of cohomological Hall algebra $\CoHA$}\label{sect rev of coha}

We review the cohomological Hall algebra $\CoHA$ following \cite{KS}.
Let $Q$ be a finite oriented quiver with the set of vertices $I$ and with $a_{ij}$ arrows from $i\in I$ to $j\in I$, so that $a_{ij}\in \bbZ _{\geq 0}$. Fix a dimension vector $\gamma =(\gamma ^i)_{i\in I}\in \bbZ _{\geq 0}^I$. We have the affine space of representations of $Q$ in the vector space $\oplus _{i\in I}\bbC ^{\gamma ^i}$:
$$M_\gamma =\prod _{i,j\in I}\bbC ^{a_{ij}\gamma ^i\gamma
^j}.$$
The variety $M_\gamma$ is acted upon via the conjugation by the group
$$G_\gamma :=\prod _{i\in I}\GL(\gamma ^i,\bbC).$$

One is interested in the $G_\gamma $-equivariant cohomology of the space $M_\gamma$. For this one can use for example the following model for the classifying space of $G_\gamma$: recall that the infinite dimensional Grassmannian
$$Gr(d,\infty):=\lim_{\stackrel{\longrightarrow}{n}}Gr(d,\bbC ^n)$$
is a classifying space for $\GL(d,\bbC)$. Then
$$BG_\gamma =\prod _{i\in I}B\!\GL (\gamma ^i,\bbC )=\prod _{i\in I}Gr(\gamma ^i,\infty)$$
We have the standard universal $G_\gamma $-bundle $EG_\gamma \to BG_\gamma$, and hence also the universal fibration over $BG_\gamma$:
$$M^{\operatorname{univ}}_\gamma:=(EG_\gamma \times M_\gamma)/G_\gamma\to EG_\gamma /G_\gamma=BG_\gamma.$$

Define the $\bbZ ^I_{\geq 0}$-graded $\bbQ$-vector space
$$\cH =\cH _Q:=\bigoplus _{\gamma \in \bbZ ^I_{\geq 0}}\cH _{\gamma},$$
where
$$\cH _\gamma :=H^\bullet _{G_\gamma }(M_\gamma ,\bbQ)=\bigoplus _{n\geq 0}H^n(M_\gamma ^{univ},\bbQ).$$

\begin{remark}
Notice that the space $M_\gamma$ is $G_\gamma$-equivariantly contractible to a point, hence $\cH _\gamma\simeq H_{G_\gamma}(pt)$. However, we are going to make $\cH $ into an algebra and the multiplication will depend on the spaces $M_\gamma$.
\end{remark}

\subsection{Product on $\Hscr$}
Fix two dimension vectors $\gamma _1,\gamma _2 \in \bbZ ^I_{\geq 0}$ and put $\gamma =\gamma _1+\gamma _2$. Consider the affine subspace $M_{\gamma _1,\gamma _2}\subset M_\gamma$, which consists of representations for which the standard subspaces $\bbC ^{\gamma _1^i}\subset \bbC ^{\gamma ^i}$ form a subrepresentation. The subspace $M_{\gamma _1,\gamma _2}$ is preserved by the action of the parabolic subgroup $G_{\gamma _1,\gamma _2}\subset G_\gamma$ which consists of transformations preserving subspaces $\bbC ^{\gamma _1^i}\subset \bbC ^{\gamma ^i}$, $i\in I$.
We have the natural morphisms of stacks
$$M_{\gamma _1}/G_{\gamma _1}\times M_{\gamma _2}/G_{\gamma _2}\stackrel{p}{\leftarrow} M_{\gamma _1,\gamma _2}/G_{\gamma _1,\gamma _2}\stackrel{i}{\to }M_{\gamma}/G_{\gamma _1,\gamma _2}\stackrel{\pi}{\to }M_{\gamma}/G_{\gamma }$$
The maps $i$ and $\pi$ are proper and we define the multiplication
$$m_{\gamma _1,\gamma _2}:\cH _{\gamma _1}\otimes \cH _{\gamma _2}\to \cH _{\gamma}$$
as the composition of the isomorphism
$$p^*:H^\bullet _{G_{\gamma _1}}(M_{\gamma _1},\bbQ)\otimes
H^\bullet _{G_{\gamma _2}}(M_{\gamma _2},\bbQ)\stackrel{\sim}{\to}H^\bullet _{G_{\gamma _1,\gamma _2}}(M_{\gamma _1,\gamma _2},\bbQ)$$
with the push forward maps $i_*$ and $\pi _*$. Explicitly, the map  $i$ is the closed embedding
$$(EG_\gamma \times M_{\gamma _1,\gamma _2})/G_{\gamma _1,\gamma _2}\to (EG_\gamma \times M_{\gamma})/G_{\gamma _1,\gamma _2}$$
and $\pi $ is the fiber bundle
$$(EG_\gamma \times M_{\gamma})/G_{\gamma _1,\gamma _2}\to
(EG_\gamma \times M_{\gamma})/G_{\gamma}$$
with the fiber $G_{\gamma}/G_{\gamma _1,\gamma _2}$ which is isomorphic to the product of Grassmannians
$\prod _{i\in I}Gr(\gamma _1^i,\bbC ^{\gamma ^i})$.
Let
$$c_1=\dim _{\bbC}M_\gamma -\dim _{\bbC}M_{\gamma _1,\gamma _2},\quad c_2=\dim _{\bbC}G_{\gamma _1,\gamma _2}-\dim _{\bbC}G_\gamma.$$
Then
$$i_*:H^\bullet _{G_{\gamma _1,\gamma _2}}(M_{\gamma _1,\gamma _2},\bbQ)\to
H^{\bullet +2c_1} _{G_{\gamma _1,\gamma _2}}(M_{\gamma},\bbQ),\quad
\pi _*:H^{\bullet} _{G_{\gamma _1,\gamma _2}}(M_{\gamma},\bbQ)\to H^{\bullet +2c_2} _{G_{\gamma}}(M_{\gamma},\bbQ ).
$$
So that
\begin{equation}\label{degree} m_{\gamma _1,\gamma _2}:H^s_{G_{\gamma _1}}(M_{\gamma _1},\bbQ)\times H^t_{G_{\gamma _2}}(M_{\gamma _2},\bbQ)\to H^{s+t+2c_1+2c_2}_{G_\gamma}(M_\gamma ,\bbQ).
\end{equation}

\begin{proposition} \label{associativity} \cite{KS} The product $m=(m_{\gamma_1,\gamma_2})_{\gamma_1,\gamma_2}$ makes $\cH$ into an associative $\bbZ _{\geq 0}^I$-graded algebra.
\end{proposition}

The $\bbZ _{\geq 0}^I$-graded algebra $(\Hscr,m)$ is called {\em cohomological Hall algebra} (CoHA).

\subsection{Explicit description of $\CoHA$}\label{explicit description of CoHA}

Let $T\subset \GL(d)$ be a maximal torus with the Lie algebra $\mathfrak{t}\simeq \bbA ^d_{\bbQ}$. Let $W=S_d=N(T)/T$ be the corresponding Weyl group. Recall that there exists a canonical isomorphism of graded algebras
\begin{equation}\label{Borel isom}   H^\bullet (B\!\GL(d),\bbQ)= (\Sym _\bbQ \mathfrak{t}^*)^W\simeq \bbQ [x_1,...,x_d]^{S_d}
\end{equation}
where $\deg(x_i)=2$.

For a dimension vector $\gamma \in \bbZ ^I_{\geq 0}$ introduce the variables $x_{i,\alpha}$, where $i\in I$ and $\alpha \in \{1,...,\gamma ^i\}$. Then we get a natural isomorphism
\begin{equation}\label{identification as symm polynomials}\cH _\gamma =\bbQ [\{x_{i,\alpha }\}_{i\in I,\alpha \in \{1,...,\gamma ^i\}}]^{\prod _{i\in I}S_{\gamma ^i}}
\end{equation}
with the multiplication \eqref{degree}
$$m_{\gamma _1,\gamma _2}:\cH _{\gamma _1}\otimes \cH _{\gamma _2}\to
\cH _{\gamma _1+\gamma _2}.$$
We will denote the product $m_{\gamma _1,\gamma _2}(f_{\gamma _1},f_{\gamma _2})$ by $f_{\gamma _1}\cdot f_{\gamma _2}$.
The following theorem gives an explicit description of the multiplication in $\cH$:

\begin{theorem} \label{formula} \cite{KS} Given two polynomials $f_1(x')\in \cH _{\gamma _1},$ $f_2(x'')\in \cH _{\gamma _2}$, their product $f_1\cdot f_2\in \cH _\gamma$, $\gamma =\gamma _1+\gamma _2$ equals to the following rational function in variables $(x'_{i,\alpha })_{i\in I,\alpha \in \{1,...,\gamma _1^i\}},$ $(x''_{i,\alpha })_{i\in I,\alpha \in \{1,...,\gamma _2^i\}}$:
\begin{equation}\label{formula-for-multiplication} \sum _{i\in I}\sum _{(\gamma _1^i,\gamma _2^i){\rm -shuffles}}f_1((x'_{i,\alpha}))f_2((x''_{i,\alpha }))\frac
{\prod_{i,j\in I}\prod_{\alpha _1=1}^{\gamma _1^i}\prod_{\alpha _2=1}^{\gamma _2^j}(x''_{j,\alpha _2}-x'_{i,\alpha _1})^{a_{ij}}}
{\prod_{i\in I}\prod_{\alpha _1=1}^{\gamma _1^i}\prod_{\alpha _2=1}^{\gamma _2^i}(x''_{i,\alpha _2}-x'_{i,\alpha _1})}.\end{equation}
\end{theorem}

\subsection{Case of a symmetric quiver}\label{case of symmetric quiver}
Let us first recall the definition of the Euler form for the quiver $Q$:
$$\chi _Q(\gamma _1,\gamma _2)=\sum _{i\in I}\gamma _1^i\gamma _2^i-\sum_{i,j\in I}a_{ij}\gamma ^i_1\gamma ^j_2.$$

For any representations $R_1,R_2$ of $Q$ with dimension vectors $\gamma _1,\gamma _2$ respectively, one has (see \cite[Corollary 1.4.3]{Brion2})
$$\sum _i(-1)^i\dim \Ext ^i(R_1,R_2)=\dim \Hom (R_1,R_2)-\dim \Ext ^1(R_1,R_2)=\chi _Q(\gamma _1,\gamma _2).$$

Note that
\begin{equation}\label{eq:chic1c2}
\chi _Q(\gamma _1,\gamma _2)=-c_1-c_2.
\end{equation}
Assume that the quiver $Q$ is symmetric, i.e. $a_{ij}=a_{ji}$, $i,j\in I$. Then the Euler form $\chi _Q(\gamma _1,\gamma _2)$ is also symmetric. In this case we can make $\cH$ a $(\bbZ _{\geq 0}^I\times \bbZ)$-graded algebra as follows: for a polymonial $f\in \cH _\gamma$ of degree $k$ we define its bidegree to be $(\gamma ,2k+\chi _Q(\gamma ,\gamma))$.
So one has
$$\cH =\bigoplus \cH _{\gamma ,l},\quad \text{where}\quad
\gamma \in{\bbZ ^I_{\geq 0}},\ l\in \bbZ.$$
It follows from \eqref{degree} and \eqref{eq:chic1c2} that the product in $\cH$ is compatible with this bigrading, i.e. we have a bigraded algebra
\begin{equation} \label{bigraded algebra}
(\cH ,\cdot),\quad m:\cH _{\gamma ,l}\otimes \cH _{\gamma ',l}\to \cH _{\gamma +\gamma ',l+l'}.
\end{equation}
The formula \eqref{formula-for-multiplication} implies that  for elements $a_{\gamma }\in \cH _{\gamma }$, $a_{\gamma '}\in \cH _{\gamma '}$ one has
\begin{equation}\label{two-products} a_{\gamma }\cdot a_{\gamma '}=(-1)^{\chi _Q(\gamma ,\gamma ')}a_{\gamma '}\cdot a_{\gamma }
\end{equation}

One can twist the multiplication by a sign so that $\cH$ becomes super-commutative with respect to the $\bbZ$-grading. Namely define the homomorphism of abelian groups $\epsilon :\bbZ ^I\to \bbZ /2\bbZ $ by the formula
$$\epsilon (\gamma)=\chi _Q(\gamma, \gamma)\mod 2.$$
Thus $\epsilon (\gamma)$ is the partity of an $\gamma$ when specializing the $\bbZ^I_{\ge 0}\times \bbZ$-grading introduced above to a $\bbZ$-grading, given by the last factor.
We have the bilinear form
\begin{equation}\label{auxillary} \bbZ ^I\times \bbZ ^I\to \bbZ /2\bbZ,\quad (\gamma _1,\gamma _2)\mapsto (\chi _Q(\gamma _1,\gamma _2)+\epsilon (\gamma _1)\epsilon (\gamma _2))\mod 2
\end{equation}
which induces a symmetric bilinear form $\beta$ on the space
$(\bbZ /2\bbZ )^I$, such that $\beta (\gamma ,\gamma)=0$ for all $\gamma \in (\bbZ /2\bbZ )^I$. Hence there exists a bilinear form $\psi $ on $(\bbZ /2\bbZ )^I$ such that
\begin{equation}\label{prop of psi} \psi (\gamma _1,\gamma _2)+\psi (\gamma _2,\gamma _1)=\beta (\gamma _1,\gamma _2).
\end{equation}

\begin{proposition}
Let $\psi$ be as in \eqref{prop of psi}. 
Define the twisted product on $\cH$ by the formula
$$a_{\gamma }\star a_{\gamma '}=(-1)^{\psi (\gamma ,\gamma ')}a_{\gamma }\cdot a_{\gamma '}\quad \text{for $a_\gamma \in \cH _\gamma $, $a_{\gamma '}\in \cH _{\gamma '}$}.$$
Then the algebra $(\cH ,\star )$ is associative and supercommutative, i.e.
$$a_{\gamma }\star a_{\gamma '}=(-1)^{\epsilon (\gamma)\epsilon (\gamma ')}a_{\gamma ' }\star a_{\gamma }.$$ Moreover, different choices of the form $\psi $ as in \eqref{prop of psi} lead to canonically isomorphic graded supercommutative algebras.
\end{proposition}

\begin{proof} 
The first claim follows easily from the definitions. To show the last claim we first construct $\psi$ that satisfies \eqref{prop of psi}. For example, choose an order on the set $I$ and define $\psi$ on elements $(e_i)_{i\in I}$ of the standard basis of $(\bbZ/2\bbZ)^I$ by
$$\psi (e_i,e_j)=\beta (e_i,e_j)\ \text{if $i>j$},\quad
\psi (e_i,e_j)=0\ \text{if $i\leq j$}.$$

Let $\psi '$ be another bilinear form on $(\bbZ /2\bbZ )^I$ such that
$$\psi '(\gamma _1,\gamma _2)+\psi '(\gamma _2,\gamma _1)=\beta (\gamma _1,\gamma _2)$$
and let
$$a_{\gamma }\star' a_{\gamma '}=(-1)^{\psi '(\gamma ,\gamma ')}a_{\gamma }\cdot a_{\gamma '}\quad \text{for $a_\gamma \in \cH _\gamma $, $a_{\gamma '}\in \cH _{\gamma '}$}$$
be the corresponding supercommutative multiplication on $\cH$.

Notice that the bilinear form $\alpha =\psi -\psi '=\psi+\psi'$ is symmetric.
Consider the standard basis $(e_i)_{i\in I}$ of the semigroup
$\bbZ _{\geq 0}^I$ and define the function $\delta :\bbZ _{\geq 0}^I\to \bbZ /2\bbZ$ by
$$\delta(0)=0,\quad\delta (e_i)=0,\quad \delta (e_{i_1}+...+e_{i_n})=\sum _{1\leq s,t\leq n}\alpha (e_{i_s},e_{i_t}) \text{ if $n\geq 2$}.$$
Then for any $\gamma ,\gamma '\in (\bbZ /2\bbZ)^I$ we have
$$\alpha (\gamma ,\gamma ')=\delta (\gamma +\gamma ')+\delta (\gamma )+\delta (\gamma ')$$
hence the map $\cH \to \cH$, $a_\gamma \mapsto (-1)^{\delta(\gamma)}a_\gamma$ defines an isomorphism of graded supercommutative algebras
\[
(\cH ,\star )\to (\cH ,\star ').\qedhere
\]
\end{proof}

\subsection{Freeness of $\CoHA$ for a symmetric quiver}
The following result was conjectured by Kontsevich-Soibelman and proved by Efimov in \cite{Ef}.

\begin{theorem} \label{efimov's main} For any symmetric quiver $Q$, the $(\bbZ ^I_{\geq 0}\times \bbZ)$-graded algebra $(\cH ,\star)$ is free super-commutative, generated by a $(\bbZ ^I_{\geq 0}\times \bbZ)$-graded vector space $V$ of the form $V=V^{prim}\otimes \bbQ [x]$, where $x$ is a variable of bidegree $(0,2)\in \bbZ ^I_{\geq 0}\times \bbZ$, and for any $\gamma \in \bbZ ^I_{\geq 0}$ the space $V^{prim}_{\gamma ,l}$ is finite dimensional and nonzero only for finitely many $l\in \bbZ$.
\end{theorem}



\subsection{Example of the quiver with one vertex}\label{first example of simpliest quiver} Consider now the quiver with one vertex and $m\geq 0$ loops. The dimension vector $\gamma $ is simply a nonnegative integer.

We have $\cH _\gamma \simeq \bbQ [x_1,...,x_\gamma]^{S_\gamma}$ and the multiplication in Theorem \ref{formula} is given as follows: if $f_1\in \cH _{\gamma _1},$ $f_2\in \cH _{\gamma _2}$, then the product $f_1\cdot f_2$ is
\begin{equation}\label{simple case}
\sum _{(\gamma _1,\gamma _2){\rm -shuffles}} f_1(x'_1,...,x'_{\gamma _1})f_2(x''_1,...,x''_{\gamma _2})
\prod_{\alpha _1=1}^{\gamma _1}\prod_{\alpha _2=1}^{\gamma _2}(x''_{\alpha _2}-x'_{\alpha _1})^{m-1}.
\end{equation}

 The formula $\eqref{simple case}$ implies that the bigraded algebra $(\cH ,\cdot)$ is commutative for odd $m$ and is supercommutative for even $m$. Moreover,
$$\chi _Q(\gamma _1,\gamma _2)=(1-m)\gamma _1\gamma _2$$
so the bilinear form $\beta $, induced from \eqref{auxillary}, is
$$\beta (\gamma _1,\gamma _2)=
(1-m)\gamma _1\gamma _2(1+(1-m)\gamma _1\gamma _2)=0$$
and therefore we may assume that the multiplications $(-)\cdot (-)$ and $(-)\star (-)$ coincide.

In case $m=0$, the bigraded algebra $\cH$ is the exterior algebra on
the graded vector space $\cH _1=\oplus _{k\geq 0} \cH _{1,2k+1}$,
where each $\cH _{1,2k+1}$ is 1-dimensional.

\subsection{Completed version of $\CoHA$}
Let $Q$ be a quiver. We define the {\it completed} cohomological Hall algebra  $\hat{\cH}=\hat{\cH}_Q$ to be
$\hat{\cH}=\bigoplus _{\gamma} \hat{\cH }_\gamma $, where
$$\hat{\cH }_\gamma :=\bigotimes _{i\in I}\prod _nH^n_{\GL(\gamma ^i)}(M_\gamma ,\bbQ)\simeq
\bigotimes _{i\in I}\prod _nH^n(B\!\GL(\gamma ^i) ,\bbQ)$$
with the multiplication
$$\hat {\cH }_\gamma \otimes \hat {\cH }_{\gamma '}\to \hat {\cH }_{\gamma +\gamma '}$$
which is induced from $(\cH ,\cdot)$. This makes  $\hat{\cH}$ an associative $\bbZ ^I_{\geq 0}$-graded algebra. Clearly $\cH$ is a graded subalgebra of $\hat{\cH}$.

\begin{remark} \label{product in completed}
Similarly to the natural isomorphism \eqref{identification as symm polynomials} we may identify the space $\hat{\cH }_\gamma$ with the tensor product of rings of symmetric power series
$$\hat{\cH }_\gamma =\bigotimes _{i\in I}\bbQ[[x_{i,1},...,x_{i,\gamma ^i}]]^{S_{\gamma ^i}}.$$
Then the formula for the product of two such functions $f_1\in \hat{\cH}_{\gamma _1},\ f_2\in \hat{\cH}_{\gamma _2}$ is identical to \eqref{formula-for-multiplication}.
\end{remark}

\begin{remark} \label{suppercommutative for completed} If the quiver $Q$ is symmetric the formula \eqref{two-products} implies that a one-sided graded ideal in $\hat{\cH}$ (or in $\cH$) is a two-sided one.
\end{remark}

\section{Review of K-theoretical Hall algebra $\KHA$}
The K-theoretical Hall algebra was studied in \cite{Pa} (in a more general setting). Let us recall the basic facts.

Let $Q$ be a quiver. Instead of considering the equivariant cohomology $H^\bullet _{G_\gamma}(M_\gamma ,\bbQ)$ we now consider the equivariant algebraic $K$-group $\cR _{\gamma }:=K^{G_\gamma}_0(M_\gamma)$, i.e. the Grothendieck group of the stack $M_\gamma /G_\gamma$, which is canonically isomorphic to the representation ring $\Rep (G_\gamma)$. Define
$$\cR =\cR _Q:=\bigoplus _{\gamma \in \bbZ ^I_{\geq 0}}\cR _\gamma.$$
As in the case of the cohomological Hall algebra $\cH$ we turn $\cR$ into a $\bbZ ^I_{\geq 0}$-graded ring with the multiplication
$$\mu _{\gamma _1,\gamma _2}:\cR _{\gamma _1}\times \cR _{\gamma _2}\to \cR _{\gamma _1+\gamma _2}$$
defined similarly to the multiplication in $\cH$. Namely, given dimension vectors $\gamma _1,\gamma _2$ and $\gamma =\gamma _1+\gamma _2$ consider the diagram of stacks
$$M_{\gamma _1}/G_{\gamma _1}\times M_{\gamma _2}/G_{\gamma _2}\stackrel{p}{\leftarrow} M_{\gamma _1,\gamma _2}/G_{\gamma _1,\gamma _2}\stackrel{i}{\to}M_\gamma /G_{\gamma _1,\gamma _2}\stackrel{\pi }{\to} M_{\gamma }/G_\gamma$$
and define the multiplication $\mu _{\gamma _1,\gamma _2}:\cR _{\gamma _1}\times \cR _{\gamma _2}\to \cR _{\gamma _1+\gamma _2}$ as the composition of the induced maps
$$\mu _{\gamma _1,\gamma _2}=\pi _*\cdot i_*\cdot p^*:K_0^{G_{\gamma _1}}(M_{\gamma _1})\times K_0^{G_{\gamma _2}}(M_{\gamma _2})\to
K_0^{G
 _\gamma }(M_\gamma).$$

\begin{proposition}\cite{Pa} \label{k0-assoc} The multiplication $\mu :\cR \otimes \cR \to \cR$ makes $\cR$ into an associative algebra over $\bbZ$.
\end{proposition}

The $\ZZ_{\geq 0}^I$-graded algebra $(\Rscr,\mu)$ is called {\em K-theoretical Hall  algebra} ($\KHA$).
\subsection{Explicit description of $\KHA$}\label{explicit kha} Let $d$ be a positive integer and let $T\simeq (\bbC ^*)^d\subset \GL(d)$ be a maximal torus, say the subgroup of diagonal matrices. Let $z_1,...,z_d$ be the standard coordinates on $T$. We have the canonical ring isomorphism $\Rep(T)=\bbZ [z_1^{\pm 1},...,z_d^{\pm 1}]$.
The action of the Weyl group $N(T)/T=S_d$ permutes the coordinates $z_i$  and the inclusion of groups $T\subset \GL(d)$ induces the isomorphism
\begin{equation}\label{isom group torus}
\Rep(\GL(d))=\Rep(T)^{S_d}=\bbZ [z_1^{\pm 1},...,z_d^{\pm 1}]^{S_d}
\end{equation}

For a quiver $Q$ and a dimension vector $\gamma =(\gamma_i)\in
\bbZ ^I$ let $z_{i,1},...,z_{i,\gamma ^i}$ be the standard coordinates on the diagonal torus in $\GL(\gamma ^i)$. Then
$$\Rep (G_\gamma)=\otimes _{\bbZ}\Rep (G_{\gamma _i})=\bbZ [\{z_{i,\alpha}^{\pm 1}\}_{i\in I,\alpha \in \{1,...,\gamma ^i\}}]^{\prod _{i\in I}S_{\gamma ^i}}.$$
Therefore we have the canonical isomorphism of groups
$$\cR _{\gamma}: =K_0^{G_\gamma}(M_\gamma )=\Rep (G_\gamma )=\bbZ [\{z_{i,\alpha}^{\pm 1}\}_{i\in I,\alpha \in \{1,...,\gamma ^i\}}]^{\prod _{i\in I}S_{\gamma ^i}}$$
In these terms the product map
$$\mu _{\gamma _1,\gamma _2}:\cR _{ \gamma _1}\times \cR _{\gamma _2}\to \cR _{ \gamma _1+\gamma _2}$$
is as follows.

\begin{theorem}\label{mult in R} \cite{Pa} In the above notation let $\gamma _1=(\gamma _1^i)$,
$\gamma _2=(\gamma _2^i)$ be dimension vectors and $f_1\in \cR _{\gamma _1}$, $f_2\in \cR _{\gamma _2}$. Then their product $\mu (f_1, f_2)\in \cR _{\gamma _1+\gamma _2}$ is the following function in the variables $(z_{i,\alpha}^{\prime \pm 1})_{i\in I,\alpha \in \{1,...,\gamma _1 ^i\}}$, $(z_{i,\alpha }^{\prime \prime \pm 1})_{i\in I,\alpha  \in \{1,...,\gamma _2 ^i\}}$
\begin{equation} \label{formula for mult in R}
\sum _{i\in I}\sum _{(\gamma _1^i,\gamma _2^i){\rm -shuffles}}f_1(z_{i,\alpha}^{\prime \pm 1})f_2(z_{i,\alpha }^{\prime \prime \pm 1})
\frac{\prod _{i,j\in I}\prod _{\alpha _1=1}^{\gamma _1^i}\prod _{\alpha _2=1}^{\gamma _2^i}(1-z^{\prime }_{i,\alpha _1}z^{\prime \prime -1}_{j,\alpha _2})^{a_{ij}}}{\prod _{i\in I}\prod _{\alpha _1=1}^{\gamma _1^i}\prod _{\alpha _2=1}^{\gamma _2^i}(1-z^{\prime }_{i,\alpha _1}z^{\prime \prime -1}_{i, \alpha _2})}.
\end{equation}
\end{theorem}

\begin{remark} \label{remark on embedding} Note that the substitution  $z_{i,\alpha}\mapsto e^{x_{i,\alpha}}$ (Chern character) defines an injective map $\cR \hookrightarrow  \hat{\cH}$ of $\bbZ _{\geq 0}^I$-graded spaces. This map is additive, but is not a homomorphism of algebras.
Using this change of variables we may and will consider $\cR _\gamma$ as a subspace of $\hat{\cH }_\gamma$. Then for $f_1 (x'_{i,\alpha})\in \cR _{\gamma _1}\subset \hat{\cH }_{\gamma _1}$, $f_2(x''_{i,\alpha})\in \cR _{\gamma _2}\subset \hat{\cH }_{\gamma _2}$ their product $f_1\cdot_{\cR} f_2:=\mu(f_1,f_2)\in \cR _{\gamma _1+\gamma _2}\subset \hat{\cH}_{\gamma _1+\gamma _2}$ is the function
\begin{equation}\label{formula for mult in k0-hall}
\sum _{i\in I}\sum _{(\gamma _1^i,\gamma _1^i){\rm -shuffles}}
 f_1(x'_{i,\alpha})f_2(x''_{i,\alpha })\frac
{\prod_{i,j\in I}\prod_{\alpha _1=1}^{\gamma _1^i}\prod_{\alpha _2=1}^{\gamma _2^j}(1-e^{x'_{i,\alpha _1}-x''_{j,\alpha _2}})^{a_{ij}}}
{\prod_{i\in I}\prod_{\alpha _1=1}^{\gamma _1^i}\prod_{\alpha _2=1}^{\gamma _2^i}(1-e^{x'_{i,\alpha _1}-x''_{i,\alpha _2}})}.
\end{equation}
Thus the formula \eqref{formula for mult in k0-hall} is obtained from the formula \eqref{formula-for-multiplication} by replacing each factor $x''-x'$ in the numerator and the denominator by the factor $(1-e^{x'-x''})$.
\end{remark}

\section{Review of twisted graded algebras following \cite{Zh}}\label{sec twist}

We review the notion of a twist of a graded algebra by an automorphism and the corresponding equivalences of categories following \cite{Zh}. Let $\Gamma$ be a semi-group (not necessarily commutative) and let
$$A=\oplus _{g\in \Gamma}A_g$$
be a $\Gamma$-graded algebra.

\begin{definition}\label{defn of a twist} \cite{Zh} A set of degree
preserving linear automorphisms of $A$, say $\sigma =\{ \sigma _g \vert g\in \Gamma \}$, is called a {\it twisting system} of $A$ if
\begin{equation}\label{twist syst} \sigma _l(\sigma _s(x)y)=\sigma  _{sl}(x)\sigma _l(y)
\end{equation}
for all $g,s,l\in \Gamma$ and $x\in A_g,y\in A_s$.
\end{definition}

Given a twisting system $\sigma$ we can define the (left) twist ${}^\sigma A$ of $A$ as a $\Gamma$-graded algebra which equals $A$ as a graded vector space and has multiplication
\begin{equation}\label{left multipl}
x\circ y=\sigma _s(x)y
\end{equation}
for all $x\in A_g,$ $y\in A_s$.

Similarly, in Definition \ref{defn of a twist} we say that $\sigma$ is a {\it right twisting system} if instead of \eqref{twist syst} the following holds
\begin{equation}\label{twist syst r} \sigma _g(y\sigma _s(z))=\sigma _g(y)\sigma _{gs}(z)
\end{equation}
for all $g,s,l\in \Gamma$ and all $y\in A_s,z\in A_l$. In this case one can similarly define the right twist $A^\sigma$ of $A$ with multiplication
\begin{equation}\label{right multiplication}
y\bullet z=y\sigma _s(z).
\end{equation}
Both ${}^\sigma A$ and $A^\sigma$ are associative algebras due to \eqref{twist syst}, \eqref{twist syst r}.

For a graded algebra $B$ denote by $B\-\Mod$ (resp. $\Mod\-B$) the abelian categories of graded left (resp. right) $B$-modules (with degree preserving morphisms).

The main point of the above twisting constructions is the following remark in \cite{Zh}.

\begin{proposition}\cite{Zh}\label{equiv of cat by twists} Let $A$ be a $\Gamma$-graded algebra and let $\sigma$ be a twisting system (resp. right twisting system) of $A$. Then there is a natural equivalence of categories $A\-\Mod\simeq {}^\sigma A\-\Mod$ (resp.
$\Mod\-A\simeq \Mod\-A^\sigma$).
\end{proposition}

\begin{proof} 
For the convenience of the reader we recall the construction of equivalences here. 
The equivalences are defined as follows (see \cite[Thm 3.1]{Zh} for details): Given a graded left (resp. right) $A$-module $M=\oplus _{g\in \Gamma}M_g$ one defines the structure of a left ${}^\sigma A$-module (resp. right $A^\sigma$-module) on $M$ by the formula
$$a\circ m=\sigma _s(a)m \quad \text{(resp. $m\bullet a=m\sigma _s(a)$)}$$
for $a\in A$ and $m\in M_s$.
\end{proof}

\begin{example} \label{example of commut gamma}Suppose that $\Gamma $ is a semigroup and $A$ is a $\Gamma $-graded algebra. Let $\Aut(A)$ denote the group of degree preserving algebra automorphisms of $A$. Suppose that we are given a homomorphism of semigroups
$$\sigma :\Gamma  \to \Aut(A),\quad g \mapsto \sigma _g.$$
Then the collection $\{ \sigma _g\}$ is a right twisting system of $A$. If the semigroup $\Gamma $ is abelian then $\{ \sigma _g\}$ is also a twisting system of $A$; hence in this case we have both the left and right twists of $A$ defined and we have the equivalences of categories in Proposition \ref{equiv of cat by twists}.
\end{example}

\section{Twists of $\hat{\cH}$}\label{twists of hat h}

Let $Q$ be a quiver. In the notation of Section \ref{sec twist} let the semi-group $\Gamma$ be the group $\bbZ ^I$. We consider the algebra $\hat{\cH}=\oplus _{\gamma \in \bbZ ^I}\hat{\cH }_{\gamma}$ as a $\bbZ ^I$-graded algebra (with $\hat{\cH}_\gamma=0$ for $\gamma\not\in \bbZ_{\geq 0}^I$) and $\hat{\cH}\-\Mod$ and $\Mod\-\hat{\cH}$ denote the categories of left and right $\bbZ ^I$-graded $\hat{\cH}$-modules respectively.

Recall that $\hat{\cH}_\gamma$ is canonically identified with the tensor product of rings of symmetric power series (cf. Remark \ref{product in completed})
$$\hat{\cH}_\gamma =\bigotimes _{i\in I}\bbQ [[(x_{i,\alpha})_{\alpha \in \{1,...,\gamma ^i\}}]]^{S_{\gamma _i}}$$
Denote
$${\bf x}_{\gamma ^i}:=x_{i,1}+...+x_{i,\gamma ^i}.$$

Assume that for each $i\in I$ we are given an additive function $l_i:\bbZ ^I \to \bbZ$. Then for $\tau  \in \bbZ ^I$ and $\gamma \in \bbZ ^I_{\geq 0}$ define
$$a^\gamma _\tau  :=\prod _{i\in I}\exp ({\bf x}_{\gamma ^i})^{l_i(\tau)}.$$
Clearly $a^\gamma _\tau$ is $S_\gamma$-invariant and hence belongs to $\hat{\cH}_\gamma$. We have
\begin{align}\label{twisting homom}
a_{\tau}^{\gamma_1}({\bf x'})a_{\tau}^{\gamma_2}({\bf x''})&=
a_{\tau}^{\gamma_1+\gamma_2}({\bf x'},{\bf x''}),\\\nonumber
a_{\tau_1}^\gamma a_{\tau_2}^\gamma&=a_{\tau_1+\tau_2}^\gamma.
\end{align}
For each $\tau \in \bbZ ^I$ define a degree preserving map $\sigma _\tau :\hat{\cH}\to \hat{\cH}$
such that
$$\sigma _\tau (f):=a_\tau ^\gamma f=\left(\prod _{i\in I}\exp ({\bf x}_{\gamma ^i})^{l_i(\tau)}\right)f\ \  \text{for $f\in \hat{\cH } _\gamma$}.$$

\begin{lemma}\label{lem:grouphom} For each $\tau \in \bbZ ^I$ the map $\sigma _\tau$ is an algebra automorphism of $\hat{\cH}$. The correspondence
$\tau \mapsto \sigma _\tau$ is a group homomorphism $\sigma :\bbZ ^I\to \Aut(\hat{\cH})$.
\end{lemma}

\begin{proof} By Remark \ref{product in completed} the product $f_1\cdot f_2$ of $f_1\in \hat{\cH}_{\gamma _1}$, $f_2\in \hat{\cH}_{\gamma _2}$ is given by the formula \eqref{formula-for-multiplication}. The first assertion now follows from the $S_\gamma$-invariance of $a^\gamma _\tau$ and the first formula in \eqref{twisting homom}. The second assertion follows from the second formula in \eqref{twisting homom}.
\end{proof}

\begin{corollary} \label{cor on existence of two twists} The homomorphism of groups $\sigma :\bbZ ^I\to \Aut(\hat{\cH })$ gives rise to the left and right twists, ${}^\sigma \hat{\cH}$ and $\hat{\cH}^\sigma$ of the $\bbZ ^I$-graded algebra $\hat{\cH}$.
\end{corollary}

\begin{proof} This follows from the fact that the group $\bbZ ^I$ is commutative as explained in Example \ref{example of commut gamma}.
\end{proof}

We will also need a small tweak of the above construction. Namely, let
$$\mu :\bbZ ^I\times \bbZ ^I\to \bbZ /2\bbZ$$
be a bilinear form (where $\bbZ/2\bbZ$ is considered with the multiplicative operation). Put
$$\tilde{a}_\tau ^\gamma :=a_\tau ^\gamma \mu (\tau ,\gamma).$$
We have an analogue of \eqref{twisting homom}
\begin{align}\label{twisting homom tilde}
\tilde{a}_{\tau}^{\gamma_1}({\bf x'})\tilde{a}_{\tau}^{\gamma_2}({\bf x''})&=
\tilde{a}_{\tau}^{\gamma_1+\gamma_2}({\bf x'},{\bf x''}),\\\nonumber
\tilde{a}_{\tau_1}^\gamma \tilde{a}_{\tau_2}^\gamma&=\tilde{a}_{\tau_1+\tau_2}^\gamma.
\end{align}
For each $\tau \in \bbZ ^I$ define the degree preserving map $\tilde{\sigma} _\tau :\hat{\cH}\to \hat{\cH}$
as follows
$$\tilde{\sigma} _\tau (f):=\tilde{a}_\tau ^\gamma f=\left(\prod _{i\in I}\exp ({\bf x}_{\gamma ^i})^{l_i(\tau)}\right)\mu(\tau ,\gamma) f\ \  \text{for $f\in \cH _\gamma$}.$$
Similarly as in Lemma \ref{lem:grouphom}, Corollary \ref{cor on existence of two twists} we obtain the following lemma.

\begin{lemma} \label{tilde-sigma} (1) For each $\tau \in \bbZ ^I$ the map $\tilde{\sigma} _\tau$ is an algebra automorphism of $\hat{\cH}$. The correspondence
$\tau \mapsto \tilde{\sigma} _\tau$ is a group homomorphism $\tilde{\sigma} :\bbZ ^I\to \Aut(\hat{\cH})$.

(2) The homomorphism of groups $\tilde{\sigma} :\bbZ ^I\to \Aut(\hat{\cH })$ gives rise to the left and right twists, ${}^{\tilde{\sigma}} \hat{\cH}$ and $\hat{\cH}^{\tilde{\sigma}}$ of the $\bbZ ^I$-algebra $\hat{\cH}$.
\end{lemma}

\section{Relation between the graded algebra $\cR$ and the twisted algebras $\hat{\cH}^\sigma$ and ${}^{\tilde{\sigma}} \hat{\cH}$ in case of a symmetric quiver}\label{rel between}
Let $a_{ij}\in \ZZ_{\geq 0}$, $a_{ij}=a_{ji}$, $i,j\in I$.
In the notation of \S \ref{twists of hat h} define the additive maps
$l_i:\bbZ ^I\to \bbZ$, $i\in I$ to be
$$l_i(\tau):=\tau ^i-\sum_{j\in I} a_{ij}\tau^j.$$
As explained in \S\ref{twists of hat h} this gives elements
$$a^\gamma _\tau  =\prod _{i\in I}\exp ({\bf x}_{\gamma ^i})^{\tau ^i-\sum_{j\in I} a_{ij}\tau^j}\in \hat{\cH}_\gamma $$
and defines a homomorphism of groups $\sigma :\bbZ ^I\to \Aut(\hat{\cH})$, where
$$\sigma _\tau (f)=a_\tau ^\gamma f=\left(\prod _{i\in I}\exp ({\bf x}_{\gamma ^i})^{\tau ^i-\sum_{j\in I} a_{ij}\tau^j}\right)f\ \  \text{for $f\in \hat{\cH }_\gamma$}.$$
By Corollary \ref{cor on existence of two twists} we obtain the  corresponding (right) twist
$(\hat{\cH}^\sigma ,\bullet)$ of the $\bbZ ^I$-graded algebra $\hat{\cH}$, where
\begin{align}\label{special case of def mult}
f_1\bullet f_2=f_1\cdot (a_{\gamma _1}^{\gamma _2}(x'')f_2)=f_1\cdot \left[\left(\prod _{i\in I}\exp ({\bf x}''_{\gamma _2 ^i})^{\gamma _1 ^i-\sum_{j\in I} a_{ij}\gamma _1^j}\right)f_2\right]
\end{align}
for $f_1(x')\in \hat{\cH}_{\gamma _1}$, $f_2(x'')\in \hat{\cH}_{\gamma _2}$.

To define a left twist of $\hat{\cH}$ we will use 
\begin{align*}
b_{\tau}^\gamma(x) &:=(a_\tau ^\gamma(x))^{-1}=\prod _{i\in I}\exp ({\bf -x}_{\gamma ^i})^{\tau ^i-\sum_{j\in I} a_{ij}\tau^j},\\
\mu (\tau ,\gamma)&:=(-1)^{\sum _{i\in I}\gamma ^i\tau ^i+\sum _{i,j\in I}a_{ij}\gamma ^i\tau ^j},\\
\tilde{b}_{\tau}^\gamma&:=b_{\tau}^\gamma\mu(\tau,\gamma)\in \hat{\cH}_\gamma. 
\end{align*}
By Lemma \ref{tilde-sigma} we obtain the corresponding (left) twist $({}^{\tilde{\sigma}} \hat{\cH},\circ)$
of the $\bbZ ^I$-graded algebra $\hat{\cH}$, where
\begin{align}\label{special case of def mult}
f_1\circ f_2=(\tilde{b}_{\gamma _2}^{\gamma _1}(x')f_1)\cdot f_2
\end{align}
for $f_1(x')\in \hat{\cH}_{\gamma _1}$, $f_2(x'')\in \hat{\cH}_{\gamma _2}$.

Our main observation is the following theorem.

\begin{theorem} \label{main thm} Assume that the quiver $Q$ is symmetric, i.e. $a_{ij}=a_{ji}$ for all $i,j\in I$. Then there exist injective degree preserving homomorphisms of $\bbZ ^I$-graded algebras
$$h^\sigma :\cR \to \hat{\cH}^\sigma,\quad {}^{\tilde{\sigma}}h:\cR \to
{}^{\tilde{\sigma}}\hat{\cH}.$$
\end{theorem}


\begin{proof} By Remark \ref{remark on embedding} we may and will consider the algebra $\cR$ as a graded subspace of $\hat{\cH}$ (via the Chern character map $ch: \cR \hookrightarrow \hat{\cH}$) where the multiplication $f_1\cdot _{\cR}f_2$ of $f_1(x'_{i,\alpha})\in \cR _{\gamma _1}$ and $f_2(x''_{i,\alpha})\in \cR _{\gamma _2}$ is given by the formula
\begin{equation}\label{mult cdot r}
\sum _{i\in I}\sum _{(\gamma _1^i,\gamma _1^i){\rm -shuffles}}
 f_1(x'_{i,\alpha})f_2(x''_{i,\alpha })\frac
{\prod_{i,j\in I}\prod_{\alpha _1=1}^{\gamma _1^i}\prod_{\alpha _2=1}^{\gamma _2^j}(1-e^{x'_{i,\alpha _1}-x''_{j,\alpha _2}})^{a_{ij}}}
{\prod_{i\in I}\prod_{\alpha _1=1}^{\gamma _1^i}\prod_{\alpha _2=1}^{\gamma _2^i}(1-e^{x'_{i,\alpha _1}-x''_{i,\alpha _2}})}.
\end{equation}

Choose a linear order on the finite set $I$ and let  $<$ be the induced lexicographic order on $I\times \ZZ_{>0}$.
For any $\gamma \in \bbZ _{\geq 0}^I$ we put
\begin{align*}\label{eq:kk}
k_\gamma((x_{i,\alpha})_{i\in I,1\leq\alpha\leq \gamma^i})&:=
\frac
{\left.\prod_{i,j\in I}\prod_{\alpha _1=1}^{\gamma^i}\prod_{\alpha _2=1}^{\gamma^j}\right|_{(j,\alpha_2)>(i,\alpha_1)}(x_{j,\alpha _2}-x_{i,\alpha _1})^{a_{ij}}}
{\left.\prod_{i\in I}\prod_{\alpha _1=1}^{\gamma^i}\prod_{\alpha _2=1}^{\gamma^i}\right|_{\alpha_2>\alpha_1}(x_{i,\alpha _2}-x_{i,\alpha _1})},\\
K_\gamma((x_{i,\alpha})_{i\in I,1\leq\alpha\leq \gamma^i})&:=
k_\gamma((e^{x_{i,\alpha}})_{i\in I,1\leq\alpha\leq \gamma^i}),\\
\eta _\gamma ((x_{i,\alpha})_{i\in I,1\leq\alpha\leq \gamma^i})&:=
\frac{k_\gamma((x_{i,\alpha})_{i\in I,1\leq\alpha\leq \gamma^i})}
{K_\gamma((x_{i,\alpha})_{i\in I,1\leq\alpha\leq \gamma^i})}.
\end{align*}
Then $\eta_\gamma$ is  $S_\gamma$-invariant 
 (as $k_\gamma$, $K_\gamma$ are invariant up to the same sign). (Notice that $\eta _\gamma$ is a power series with the constant coefficient 1.)

When we use the notation $k_{\gamma _1+\gamma _2}((x'_{i,\alpha})_{i\in I,1\leq\alpha\leq \gamma^i_1},(x''_{i,\alpha})_{i\in I,1\leq\alpha\leq \gamma^i_2})$ it is understood that for a fixed index $i\in I$ the variables are ordered as $(x'_{i,1},...,x'_{i,\gamma ^i_1},x''_{i,1},...,x''_{i,\gamma ^i_2})$. Similarly for $K_{\gamma _1+\gamma _2}$ and $\eta _{\gamma _1+\gamma _2}$.

Moreover, let us denote
\begin{align*}
k_{\gamma_1,\gamma_2}((x'_{i,\alpha})_{i\in I,1\leq\alpha\leq \gamma_1^i},(x''_{i,\alpha})_{i\in I,1\leq\alpha\leq \gamma_2^i})&:=
\frac
{\prod_{i,j\in I}\prod_{\alpha _1=1}^{\gamma _1^i}\prod_{\alpha _2=1}^{\gamma _2^j}(x''_{j,\alpha _2}-x'_{i,\alpha _1})^{a_{ij}}}
{\prod_{i\in I}\prod_{\alpha _1=1}^{\gamma _1^i}\prod_{\alpha _2=1}^{\gamma _2^i}(x''_{i,\alpha _2}-x'_{i,\alpha _1})},\\
K_{\gamma_1,\gamma_2}((x'_{i,\alpha})_{i\in I,1\leq\alpha\leq \gamma_1^i},(x''_{i,\alpha})_{i\in I,1\leq\alpha\leq \gamma_2^i})&:=
k_{\gamma_1,\gamma_2}((e^{x'_{i,\alpha}})_{i\in I,1\leq\alpha\leq \gamma_1^i},(e^{x''_{i,\alpha}})_{i\in I,1\leq\alpha\leq \gamma_2^i}).
\end{align*}
We further denote
\begin{multline*}
L_{\gamma_1,\gamma_2}((x'_{i,\alpha})_{i\in I,1\leq\alpha\leq \gamma_1^i},(x''_{i,\alpha})_{i\in I,1\leq\alpha\leq \gamma_2^i}):=
\frac
{\prod_{i,j\in I}\prod_{\alpha _1=1}^{\gamma _1^i}\prod_{\alpha _2=1}^{\gamma _2^j}(1-e^{x'_{i,\alpha _1}-x''_{j,\alpha _2}})^{a_{ij}}}
{\prod_{i\in I}\prod_{\alpha _1=1}^{\gamma _1^i}\prod_{\alpha _2=1}^{\gamma _2^i}(1-e^{x'_{i,\alpha _1}-x''_{i,\alpha _2}})}.
\end{multline*}

When it is clear from the context we  slightly abuse the notation and write  $k_{\gamma_1,\gamma_2}$ (or $k_{\gamma_1,\gamma_2}(x',x'')$)  for $k_{\gamma_1,\gamma_2}((x'_{i,\alpha})_{i\in I,1\leq\alpha\leq \gamma_1^i},(x''_{i,\alpha})_{i\in I,1\leq\alpha\leq \gamma_2^i})$. Similarly for $K_{\gamma_1,\gamma_2}$, $L_{\gamma_1,\gamma_2}$.

\medskip

Define the additive map $h^{\sigma}:\cR \to \hat{\cH}^{\sigma}$ to be:
$$f\mapsto \eta_\gamma f,\quad \text{for}\quad f\in \cR _\gamma \subset \hat{\cH}_\gamma$$
We claim that the map $h^{\sigma}$ is a (injective) homomorphism of graded algebras
$$h^{\sigma}:\cR \to \hat{\cH}^{\sigma}.$$
Indeed, let $f_1((x'_{i,\alpha})_{i\in I,1\leq\alpha\leq \gamma_1^i})\in \cR _{\gamma _1}$, $f_2((x''_{i,\alpha})_{i\in I,1\leq\alpha\leq \gamma_2^i})\in \cR _{\gamma _2}$. By Remark \ref{remark on embedding} we need to check that

\begin{equation}\label{eq:tocheck}
(\eta_{\gamma_1} f_1)\bullet
(\eta_{\gamma_2} f_2)
=
\eta_{\gamma_1+\gamma_2}(f_1\cdot _{\cR} f_2).
\end{equation}

The LHS of \eqref{eq:tocheck} is by definition equal to
$$
\sum_{l\in I}\sum_{(\gamma_1^l,\gamma_2^l)-\text{-shuffles}}(\eta_{\gamma_1} f_1)
(\eta_{\gamma_2} f_2)
k_{\gamma_1,\gamma_2}
a_{\gamma _1}^{\gamma _2}(x'')
$$
which by Lemma \ref{lem:manipul} below (by \eqref{eq:mult-for-eta}, \eqref{eq:KKLL}, respectively) and the $S_{\gamma_1+\gamma_2}$-invariance of $\eta_{\gamma_1+\gamma_2}$ equals 
\begin{align*}  \sum_{l\in I}&\sum_{(\gamma _1^l,\gamma _2^l)\text{-shuffles}}\eta_{\gamma_1+\gamma_2}
f_1
f_2K_{\gamma_1,\gamma_2}
a_{\gamma _1}^{\gamma _2}(x'')\\
&= \eta_{\gamma_1+\gamma_2}
\sum_{l\in I}\sum_{(\gamma^l _1,\gamma^l _2)\text{-shuffles}}f_1f_2L_{\gamma_1,\gamma_2}
a_{\gamma _1}^{\gamma _2}(x'')(a_{\gamma _1}^{\gamma _2}(x''))^{-1}
\\
&=\eta_{\gamma_1+\gamma_2}\sum_{l\in I}\sum_{(\gamma^l _1,\gamma^l _2)\text{-shuffles}}f_1f_2L_{\gamma_1,\gamma_2}\\
&=\eta_{\gamma_1+\gamma_2}(f_1\cdot _{\cR} f_2)
\end{align*}
This proves \eqref{eq:tocheck}.

\medskip

The construction of the homomorphism  ${}^{\tilde{\sigma}} h$ is similar and we only sketch it.
In fact we only need a small modification of the previous argument. Similarly to the functions $k_\gamma$, $K_\gamma$ ... we define the new functions:
$$\tilde{k}_{\gamma}(x):=k_{\gamma}(x);\quad \tilde{K}_{\gamma}(x):=K_{\gamma}(-x);\quad
\tilde{\eta}_\gamma (x):=\tilde{k}_\gamma (x)/\tilde{K}_{\gamma} (x)$$
$$\tilde{k}_{\gamma _1,\gamma _2}(x',x''):=k_{\gamma _1,\gamma _2}(x',x'');\ \ \tilde{K}_{\gamma _1,\gamma _2}(x',x''):=K_{\gamma _1,\gamma _2}(-x',-x'')$$
$$\tilde{L}_{\gamma _1,\gamma _2}(x',x''):=L_{\gamma _1,\gamma _2}(x',x'').$$

Define an additive map ${}^{\tilde{\sigma}}h:\cR \to {}^{\tilde{\sigma}}\hat{\cH}$ as
$$f\mapsto \tilde{\eta }_\gamma f,\quad f\in \cR _\gamma.$$
Exactly as above (using Lemma \ref{lem:manipul'} below instead of Lemma \ref{lem:manipul}) one proves that ${}^{\tilde{\sigma}}h$ is a homomorphism of graded algebras. This completes the proof of the theorem.
\end{proof}
\begin{lemma}\label{lem:manipul}
The above functions are connected as follows:
\begin{align}\label{eq:mult-for-eta}
\eta_{\gamma_1+\gamma_2}((x'_{i,\alpha})_{i\in I,1\leq\alpha\leq \gamma_1^i},&(x''_{i,\alpha})_{i\in I,1\leq\alpha\leq \gamma_2^i})\\&=\nonumber
\eta_{\gamma_1} ((x'_{i,\alpha})_{i\in I,1\leq\alpha\leq \gamma_1^i})\eta_{\gamma_2}((x''_{i,\alpha})_{i\in I,1\leq\alpha\leq \gamma_2^i})
\frac{k_{\gamma_1,\gamma_2}}{K_{\gamma_1,\gamma_2}},
\end{align}

\begin{equation}\label{eq:KKLL}
K_{\gamma_1,\gamma_2}((x'_{i,\alpha})_{i\in I,1\leq\alpha\leq \gamma_1^i},(x''_{i,\alpha})_{i\in I,1\leq\alpha\leq \gamma_2^i})
=
L_{\gamma_1,\gamma_2}(a_{\gamma _1}^{\gamma _2}(x''))^{-1}.
\end{equation}
\end{lemma}

\begin{proof} To see \eqref{eq:mult-for-eta}, consider first the ratio
\begin{equation}\label{ratio}
\frac{k_{\gamma _1+\gamma _2}(x',x'')}{k_{\gamma _1}(x')k_{\gamma _2}(x'')}
\end{equation}
which is equal to the triple product
\begin{multline}\label{auxillary comp}
\left[\prod_{i\in I}\prod _{\alpha _1=1}^{\gamma _1^i}\prod _{\alpha _2=1}^{\gamma _2^i}(x''_{i,\alpha _2}-x'_{i,\alpha _1})^{a_{ii}-1}\right]\cdot \\\left[\prod_{j>i}\prod _{\alpha _1=1}^{\gamma _1^i}\prod _{\alpha _2=1}^{\gamma _2^j}(x''_{j,\alpha _2}-x'_{i,\alpha _1})^{a_{ij}}\right] \cdot \left[\prod_{j>i}\prod _{\alpha _1=1}^{\gamma _1^j}\prod _{\alpha _2=1}^{\gamma _2^i}(x'_{j,\alpha _1}-x''_{i,\alpha _2})^{a_{ij}}\right]
\end{multline}
Components $(x''-x')$ of the first two factors appear also in $k_{\gamma _1,\gamma _2}$, whereas instead of components $(x'_{j,\alpha _1}-x''_{i,\alpha _2})^{a_{ij}}$ of the third factor we see
$(x''_{i,\alpha _2}-x'_{j,\alpha _1})^{a_{ji}}$ in $k_{\gamma _1,\gamma _2}$. Because $a_{ij}=a_{ji}$ we conclude that the ratio \eqref{ratio} differs from $k_{\gamma _1,\gamma _2}(x',x'')$ by the sign factor
$$\prod_{j>i}\prod _{\alpha _1=1}^{\gamma _1^j}\prod _{\alpha _2=1}^{\gamma _2^i}(-1)^{a_{ij}}$$
The same sign factor appears in the comparison of the ratio
$$\frac{K_{\gamma _1+\gamma _2}(x',x'')}{K_{\gamma _1}(x')K_{\gamma _2}(x'')}$$
and $K_{\gamma _1,\gamma _2}(x',x'')$. This proves the equality
\eqref{eq:mult-for-eta}.

For the equality \eqref{eq:KKLL} notice that
$$
\begin{array}{rcccl}
K_{\gamma_1,\gamma_2}/L_{\gamma_1,\gamma_2} & = & \frac
{\prod_{i,j\in I}\prod_{\alpha _1=1}^{\gamma _1^i}\prod_{\alpha _2=1}^{\gamma _2^j}(e^{x''_{j,\alpha _2}})^{a_{ij}}}
{\prod_{i\in I}\prod_{\alpha _1=1}^{\gamma _1^i}\prod_{\alpha _2=1}^{\gamma _2^i}e^{x''_{i,\alpha _2}}}
& = &
\frac
{\prod_{i,j\in I}\prod_{\alpha _2=1}^{\gamma _2^j}(e^{x''_{j,\alpha _2}})^{a_{ij}\gamma ^i_1}}
{\prod_{i\in I}\prod_{\alpha _2=1}^{\gamma _2^i}(e^{x''_{i,\alpha _2}})^{\gamma ^i_1}}
\end{array}
$$
$$
\begin{array}{rcccl}
 & = &
\frac
{\prod_{i,j\in I}\prod_{\alpha _2=1}^{\gamma _2^i}(e^{x''_{i,\alpha _2}})^{a_{ji}\gamma ^j_1}}
{\prod_{i\in I}\prod_{\alpha _2=1}^{\gamma _2^i}(e^{x''_{i,\alpha _2}})^{\gamma ^i_1}} & = & \prod_{i\in I}\prod_{\alpha _2=1}^{\gamma _2^i}(e^{x''_{i,\alpha _2}})^{-\gamma ^i_1+\sum _ja_{ji}\gamma ^j_1}.
\end{array}
$$
Since we assumed that $a_{ij}=a_{ji}$ the equality \eqref{eq:KKLL} follows.
\end{proof}


\begin{lemma}\label{lem:manipul'}
The above functions are connected as follows:
\begin{align}\label{eq:mult-for-eta-tilde}
\tilde{\eta}_{\gamma_1+\gamma_2}(x',x'')=
\tilde{\eta}_{\gamma_1} (x')\tilde{\eta}_{\gamma_2}(x'')
\frac{\tilde{k}_{\gamma_1,\gamma_2}(x',x'')}{\tilde{K}_{\gamma_1,\gamma_2}(x',x'')},
\end{align}

\begin{equation}\label{eq:KKLL-tilde}
\tilde{K}_{\gamma_1,\gamma_2}(x',x'')
=\tilde{L}_{\gamma _1,\gamma _2}(x',x'')(\tilde{b}_{\gamma _2}^{\gamma _1})^{-1}.
\end{equation}
\end{lemma}

\section{Categories of locally finite modules and their equivalences}\label{cat of loc finite}

Although some of the material in this section makes sense for general quivers, we assume for simplicity that the quiver $Q$ is symmetric, so that the results of Section \ref{rel between} can be applied. In particular we use the notation ${}^{\tilde{\sigma}}\hat{\cH}$ and $\hat{\cH}^\sigma$ from Section \ref{rel between}.

For $\gamma \in \bbZ _{\geq 0}^I$ and $n\in \bbZ_{\geq 0}$ denote by $\cH _\gamma ^{\geq n}\subset \cH _\gamma$ (resp. $\hat{\cH} _\gamma ^{\geq n}\subset \hat{\cH} _\gamma$, resp. ${}^{\tilde{\sigma}}\hat{\cH} _\gamma ^{\geq n}\subset {}^{\tilde{\sigma}}\hat{\cH} _\gamma$, resp.
$\hat{\cH} _\gamma ^{\sigma ,\geq n}\subset \hat{\cH}^\sigma _\gamma$) the space of polynomials (resp. of power series) without monomials of degree $<n$.

Given a $\bbZ ^I$-graded $\cH$- (resp. $\hat{\cH}$-, resp. $\hat{\cH}^\sigma$- ...) module $M=\oplus_\gamma M_\gamma$ and $\tau \in \bbZ ^I$ we define the twist $M(\tau)$ to be the graded module such that $M(\tau)_\gamma =M_{\tau -\gamma}$.

\begin{definition} \label{def of large ideal and loc f module} A graded (left or right) ideal $I\subset \cH$ is {\it large} if for every $\gamma \in \bbZ _{\geq 0}^I$ there exists $n_\gamma \in \bbZ _{\geq 0}$ such that $I_\gamma \supset \cH _\gamma ^{\geq n_\gamma}$.

A graded (left or right) $\cH$-module $M$ is {\it locally finite} if there exists a surjection of graded modules
$$\theta :\bigoplus \cH (\tau)\to M$$
such that for every summand $\cH (\tau)$, $\ker \theta \vert _{\cH (\tau)}$ is the twist of a large ideal in $\cH $.
Denote by $\cH\-\Mod_{lf}\subset \cH\-\Mod$ be the full subcategory of locally finite modules.

Similarly one defines large ideals and locally finite graded $\hat{\cH}$-, $\hat{\cH}^\sigma$-, ${}^{\tilde{\sigma}}\hat{\cH}$-modules.
\end{definition}

\begin{lemma} \label{first lemma on equiv}(1) The inclusion of graded algebras $\cH \subset \hat{\cH}$ induces by restriction of scalars the equivalence of categories
$$\cH\-\Mod_{lf}\simeq \hat{\cH}\-\Mod_{lf}.$$

(2) The equivalences of categories $\hat{\cH}\-\Mod\simeq {}^{\sigma}\hat{\cH}\-\Mod$, $\hat{\cH} \-\Mod\simeq \hat{\cH}^\sigma \-\Mod$ ... from Proposition \ref{equiv of cat by twists} induce the equivalences of the corresponding categories of locally finite modules $$\hat{\cH} \-\Mod_{lf}\simeq {}^{\sigma}\hat{\cH} \-\Mod_{lf},\quad \hat{\cH} \-\Mod_{lf}\simeq \hat{\cH}^\sigma  \-\Mod_{lf} $$
$$\hat{\cH} \-\Mod_{lf}\simeq {}^{\tilde{\sigma}}\hat{\cH} \-\Mod_{lf},\quad \hat{\cH} \-\Mod_{lf}\simeq \hat{\cH}^{\tilde{\sigma}}  \-\Mod_{lf}. $$
\end{lemma}

\begin{proof} (1) Given a large graded ideal $\hat{I}\subset \hat{\cH}$ the ideal $I:=\hat{I}\cap \cH$ is also large. Moreover, the natural
map $\cH/I\to \hat{\cH}/\hat{I}$ is an isomorphism. This implies that a locally finite $\hat{\cH}$-module gives by restriction of scalars a locally finite $\cH$-module. So we get the fully faithful functor
\begin{equation}
\label{restr of scal} \hat{\cH} \-\Mod_{lf}\to \cH \-\Mod_{lf}.
\end{equation}
We need to prove that this functor is essentially surjective.
Let $I\subset \cH$ be a large ideal. Let $\hat{I}_\gamma \subset \hat{\cH}_\gamma$ be the completion of $I_\gamma$ in the adic topology of the power series ring $\hat{\cH }_\gamma$.
Then $\hat{I}:=\oplus _\gamma \hat{I}_\gamma $ is a large graded ideal in $\hat{\cH}$ (because the product map $\hat{\cH} _\gamma \times \hat{\cH}_\tau \to \hat{\cH}_{\gamma +\tau}$ is continuous). Moreover (since $I$ is large) we have $\hat{I}\cap \cH =I$ and the natural map $\cH/I\to \hat{\cH}/\hat{I}$ is an isomorphism. This implies that any locally finite $\cH$-module is in fact a locally finite $\hat{\cH}$-module and proves the essential surjectivity of the functor \eqref{restr of scal}.

(2) This is clear.
\end{proof}

We now repeat the above definitions for the graded algebra $\cR$.
For each $\gamma \in \bbZ _{\geq 0}^I$ consider $\cR _\gamma$ as the space of symmetric polynomials of the variables $z$ (Section \ref{explicit kha}).
Let $\cR _\gamma ^{\geq n}\subset \cR _\gamma$ be the subspace of polynomials divisible by some $(z_{i_1}-1)\ldots (z_{i_d}-1)$. Note that under the embedding $\cR \subset \hat{\cH}$ we have $\cR _\gamma^{\geq n} =\cR \cap \hat{\cH}_\gamma ^{\geq n}$ (and similarly for the ring homomorphisms ${}^{\tilde{\sigma}}h:\cR \to {}^{\tilde{\sigma}}\hat{\cH}$ and $h^\sigma :\cR \to \hat{\cH}^\sigma$). We say that a graded (left or right) ideal $J\subset \cR$ is {\it large} if for every $\gamma $ there exists $n_\gamma \in \bbZ _{\geq 0}$ such that $J_\gamma \supset \cR _\gamma ^{\geq n_\gamma}$. It follows that the inverse image under ${}^{\tilde{\sigma}}h$ or $h^\sigma$ of a large ideal is also large. Exactly as in Definition
\ref{def of large ideal and loc f module} we define the notion of a
locally finite graded $\cR$-module. Let
$$\cR\-\Mod_{lf}\subset \cR\-\Mod$$
be the full subcategory of locally finite $\cR$-modules. For a better statement in the next lemma we consider the algebra $\cR _\bbQ:=\cR \otimes _\bbZ \bbQ$ with the obvious extension of the above definitions.

Similarly to Lemma
\ref{first lemma on equiv} we have the following result.

\begin{lemma}\label{second lemma on equiv}
The homomorphisms of graded algebras ${}^{\tilde{\sigma}}h:\cR \to {}^{\tilde{\sigma}}\hat{\cH}$ and $h^\sigma :\cR \to \hat{\cH}^\sigma$ (Theorem \ref{main thm}) induce by restriction of scalars the equivalences of categories
$${}^{\tilde{\sigma}}\hat{\cH}\-\Mod_{lf}\simeq \cR_{\bbQ } \-\Mod_{lf},\quad \hat{\cH}^\sigma\-\Mod_{lf}\simeq \cR _\bbQ \-\Mod_{lf}$$
and similarly for right modules.
\end{lemma}

\begin{proof} Exactly the same as that of Lemma \ref{first lemma on equiv}.
\end{proof}

\begin{corollary} \label{summary of equivalences} For a symmetric quiver we have the equivalences of categories
$$\cH\-\Mod_{lf}\simeq \hat{\cH}\-\Mod_{lf}\simeq
{}^{\tilde{\sigma}}\hat{\cH}\-\Mod_{lf}\simeq \cR _\bbQ\-\Mod_{lf}.
$$
\end{corollary}

\begin{proof} This follows from Lemmas \ref{first lemma on equiv},
\ref{second lemma on equiv} and Proposition \ref{equiv of cat by twists}.
\end{proof}

\begin{example} \label{explicit module structure}
Let $M$ be a $\bbZ ^I$-graded cyclic locally finite $\hat{\cH}$- (or $\cH$-) module with a surjection of $\hat{\cH}$-modules $\overline{(-)}:\hat{\cH}\to M$. We want to describe explicitly the  corresponding graded $\cR$-module structure on
$M$ given by Corollary \ref{summary of equivalences}.

Let $f_1\in \hat{\cH}_\gamma$, $f_2\in \hat{\cH}_\tau$ and $m=\overline{f_2}\in M_\tau$.
Then
$$f_1m=\overline{f_1\cdot f_2}$$
where $f_1\cdot f_2$ is the product in $\hat{\cH}$. So the ${}^{\tilde{\sigma}}\hat{\cH}$-module structure on $M$ is as follows:
$$f_1\circ m=\overline{f_1\circ f_2}=\overline{(\tilde{b}_\tau ^\gamma f_1)\cdot f_2}.$$
Finally, considering $\cR$ as a subspace of $\hat{\cH}$ as before (via the Chern character map), for $f\in \cR _\gamma$ one has
$$fm={}^{\tilde{\sigma}}h(f)\circ m=(\tilde{\eta}_\gamma f)\circ m=
\overline{(\tilde{b}_\tau ^\gamma \tilde{\eta}_\gamma f)\cdot f_2}.$$
\end{example}

\section{Example: Cohomology and K-theory of framed moduli as cyclic locally finite modules over $\cH$ and $\cR$}

\subsection{} We recall the graded $\cH$- and $\cR$-modules corresponding to framed moduli of $Q$. A good reference is \cite{Fr}.

Let $Q$ be a finite directed quiver with the set of vertices $I$ and $a_{ij}$ arrows from vertex $i$ to vertex $j$. Consider the augmented quiver $\tilde{Q}$, which is obtained from $Q$ by adding a vertex $*$ and one arrow from $*$ to any vertex in $Q$, i.e. $a_{*i}=1$, $a_{i*}=0$ for any $i\in I$.

One uses the augmented quiver to construct moduli spaces of "framed" representations of $Q$ with a dimension vector $\gamma$. More precisely, consider the augmented dimension vector $\tilde{\gamma}:=(1^*,\gamma)$ of $\tilde{Q}$. Denote by $\tilde{M}_{\gamma}$  the space of $\tilde{Q}$-representations with dimension vector $\tilde{\gamma}$. It  consists of a $Q$-representation $(V^i)_{i\in I}$ and a choice of a vector $v\in \oplus V^i$. The action of $G_\gamma$ on $M_\gamma$ extends naturally to an action on $\tilde{M}_{\gamma}$.

Choose the stability condition
$$\theta :\{\{*\}\cup I\}\to \bbZ, \quad \theta (*)=-\sum _i\gamma ^i,\quad \text{and}\quad \theta (i)=1\quad \text{for all}\quad i\in I$$
so that $\theta (\tilde{\gamma})=0$.
Recall that a representation $\tilde{V}$ of $\tilde{Q}$ with dimension vector $\tilde{\gamma}$ is $\theta$-semi-stable (resp. stable) if $\theta (\tilde{\gamma})=0$ and for every subrepresentation $\tilde{V}'\subset \tilde{V}$ with dimension vector $\tilde{\gamma}'$ one has $\theta (\tilde{\gamma}')\geq 0$ (resp. $>0$ for all proper subrepresentations $N'$). Actually in this case a representation is stable if it is semi-stable. We denote by $\tilde{M}_\gamma^{st}\subset \tilde{M}_{\gamma}$ the open subset consisting of $\theta$-stable representations of $\tilde{Q}$. Obviously a representation $\tilde{V}=(v\in \oplus V^i)\in \tilde{M}_{\gamma}$ is stable if and only if the vector $v$ generates the $Q$-representation $(V^i)$.

The $G_\gamma $-action on  $\tilde{M}_{\gamma}$ restricts to an action on $\tilde{M}_\gamma^{st}$. Moreover, the $G_\gamma$-action on $\tilde{M}_\gamma^{st}$ is free and there exists a geometric quotient
$$\tilde{M}_{\gamma }^{st}\to \tilde{M}_{\gamma }^{st}\quot G_\gamma=:\cM_\gamma$$
which is a principal $G_\gamma$-bundle. The quotient $\cM _\gamma$  
 is a smooth quasi-projective variety (when nonempty).
It is the moduli space of representations of $Q$ with dimension vector $\gamma$ and a choice of a generator.

Forgetting the generator gives
the map
$$\tau :\tilde{M}_\gamma \to M_\gamma$$
which is a $G_\gamma$-equivariant vector bundle.
Like the graded space $\cH =\oplus _\gamma H_{G_\gamma}(M_\gamma ,\bbQ)$ is naturally a $\bbZ _{\geq 0}^I$-graded $\cH$-module, so is the
the space
$$\bigoplus _\gamma H_{G_\gamma}(\tilde{M}_\gamma ,\bbQ).$$
Indeed, let $\gamma _1,\gamma _2\in \bbZ ^I_{\geq 0}$ and put $\gamma =\gamma _1+\gamma _2$. In the notation of Section \ref{sect rev of coha} let
$$\tilde{M}_\gamma \supset \tilde{M}_{\gamma _1,\gamma _2}:=\tau^{-1}(M_{\gamma _1,\gamma _2})$$
We have the obvious commutative diagram of stacks
\begin{equation}\label{ext of module}
\begin{array}{ccccccc}M_{\gamma _1}/G_{\gamma _1}\times \tilde{M}_{\gamma _2}/G_{\gamma _2} & \stackrel{p}{\leftarrow} & \tilde{M}_{\gamma _1,\gamma _2}/G_{\gamma _1,\gamma _2} & \stackrel{i}{\to } & \tilde{M}_{\gamma}/G_{\gamma _1,\gamma _2} & \stackrel{\pi}{\to } & \tilde{M}_{\gamma}/G_{\gamma }\\
\downarrow id\times \tau & & \downarrow \tau & & \downarrow \tau & & \downarrow \tau\\
M_{\gamma _1}/G_{\gamma _1}\times M_{\gamma _2}/G_{\gamma _2} & \stackrel{p}{\leftarrow} & M_{\gamma _1,\gamma _2}/G_{\gamma _1,\gamma _2} & \stackrel{i}{\to } & M_{\gamma}/G_{\gamma _1,\gamma _2} & \stackrel{\pi}{\to } & M_{\gamma}/G_{\gamma }
\end{array}
\end{equation}
Taking the cohomology of stacks in the upper row of the diagram \eqref{ext of module} gives the map
$$H_{G_{\gamma _1}}(M_{\gamma _1})\times H_{G_{\gamma _2}}(\tilde{M}_{\gamma _2})\stackrel{\pi _*\cdot i_*\cdot p^*}{\longrightarrow} H_{G_{\gamma}}(\tilde{M}_\gamma)$$
which defines the graded $\cH$-module structure on $\oplus _\gamma H_{G_\gamma}(\tilde{M}_\gamma ,\bbQ)$. Moreover, the equivariant vector bundle $\tau :\tilde{M}_\gamma \to M_\gamma$ induces an isomorphism
$$\tau ^*:\cH _\gamma =H_{G_{\gamma}}(M_\gamma)\stackrel{\sim}{\to } H_{G_{\gamma}}(\tilde{M}_\gamma)$$
which identifies the $\cH$-modules $\cH\simeq \oplus _\gamma H_{G_\gamma}(\tilde{M}_\gamma ,\bbQ)$.

Similarly, the space $\oplus _\gamma H_{G_\gamma}(\tilde{M}^{st}_\gamma ,\bbQ)$ is a graded $\cH$-module. Indeed, in the previous notation define
$$\tilde{M}^{st}_{\gamma _1,\gamma _2}:=\tilde{M}_{\gamma _1,\gamma _2}\cap
\tilde{M}^{st}_{\gamma }.$$
Then the projection $\tilde{M}_{\gamma _1,\gamma _2}\stackrel{p}{\to} M_{\gamma _1}\times \tilde{M}_{\gamma _2}$ restricts to the projection
$$\tilde{M}^{st}_{\gamma _1,\gamma _2}\stackrel{p}{\to} M_{\gamma _1}\times \tilde{M}^{st}_{\gamma _2}$$ and we have the commutative diagram of stacks
\begin{equation}\label{ext of module to stable}
\begin{array}{ccccccc}M_{\gamma _1}/G_{\gamma _1}\times \tilde{M}^{st}_{\gamma _2}/G_{\gamma _2} & \stackrel{p}{\leftarrow} & \tilde{M}^{st}_{\gamma _1,\gamma _2}/G_{\gamma _1,\gamma _2} & \stackrel{i}{\to } & \tilde{M}^{st}_{\gamma}/G_{\gamma _1,\gamma _2} & \stackrel{\pi}{\to } & \tilde{M}^{st}_{\gamma}/G_{\gamma }\\
\downarrow id\times j & & \downarrow j & & \downarrow j & & \downarrow j\\
M_{\gamma _1}/G_{\gamma _1}\times \tilde{M}_{\gamma _2}/G_{\gamma _2} & \stackrel{p}{\leftarrow} & \tilde{M}_{\gamma _1,\gamma _2}/G_{\gamma _1,\gamma _2} & \stackrel{i}{\to } & \tilde{M}_{\gamma}/G_{\gamma _1,\gamma _2} & \stackrel{\pi}{\to } & \tilde{M}_{\gamma}/G_{\gamma }
\end{array}
\end{equation}
where $j:\tilde{M}_\gamma^{st}\hookrightarrow \tilde{M}_\gamma$ is the open embedding.
As before we obtain the map
$$H_{G_{\gamma _1}}(M_{\gamma _1})\times H_{G_{\gamma _2}}(\tilde{M}^{st}_{\gamma _2})\stackrel{\pi _*
\cdot i_*\cdot p^*}{\longrightarrow} H_{G_{\gamma}}(\tilde{M}^{st}_\gamma)$$
which endows the space
$$\tilde{\cH}^{st}:=\bigoplus _\gamma H_{G_{\gamma}}(\tilde{M}_\gamma ^{st},\bbQ)$$
with a structure of a $\bbZ _{\geq 0}^I$-graded $\cH$-module. It comes with a morphism of $\cH$-modules
$$j^*:\cH \to \tilde{\cH}^{st}$$
induced by the open embedding $j:\tilde{M}^{st}_\gamma\hookrightarrow \tilde{M}_\gamma$.

\begin{proposition} \label{prop that geom mod is locally finite}The morphism of graded $\cH$-modules $j^*:\cH \to \tilde{\cH}^{st}$ is surjective and its kernel is a large ideal in $\cH$. In particular, $\tilde{\cH}^{st}$ is a cyclic locally finite $\cH$-module.
\end{proposition}
\begin{proof} We use the ideas from \cite{Fr}. Fix $\gamma \in \bbZ_{\geq 0}^I$. For a smooth complex algebraic variety $Y$ with a $G_\gamma $-action one has the equivariant Chow group $A_{G_\gamma}(Y)$ \cite{EdGra}. It comes together with the cycle map
to equivariant cohomology $A_{G_\gamma}(Y)\to H_{G_\gamma}(Y,\bbQ)$. The open embedding $j:\tilde{M}^{st}_\gamma\hookrightarrow \tilde{M}_\gamma$ and the quotient map $q:\tilde{M}^{st}_\gamma \to \cM _\gamma$ give rise to a commutative diagram
\begin{equation}\label{equivar chow}
\begin{array}{ccccc}
A(\cM_\gamma)_\bbQ & \stackrel{q^*}{\rightarrow} & A_{G_\gamma}(\tilde{M}^{st}_\gamma)_\bbQ& \stackrel{j^*}{\leftarrow} & A_{G_\gamma}(\tilde{M}_\gamma)_\bbQ\\
\downarrow & & \downarrow & & \downarrow\\
H(\cM _\gamma ,\bbQ) & \stackrel{q^*}{\to} & H_{G_\gamma}(\tilde{M}^{st}_\gamma,\bbQ) & \stackrel{j^*}{\leftarrow} & H_G(\tilde{M}_\gamma,\bbQ)
\end{array}
\end{equation}
where the maps $q^*$ are isomorphisms, since $q$ is a principal $G_\gamma$-bundle.

The map $A_{G_\gamma}(\tilde{M}^{st}_\gamma)_\bbQ  \stackrel{j^*}{\leftarrow}  A_{G_\gamma}(\tilde{M}_\gamma)_\bbQ$ is clearly surjective and the right vertical arrow is an isomorphism. So to prove that the map $H_{G_\gamma}(\tilde{M}^{st}_\gamma,\bbQ)  \stackrel{j^*}{\leftarrow}  H_G(\tilde{M}_\gamma,\bbQ)$ is surjective it suffices to show that the equivariant cycle map $A^*_{G_\gamma}(\tilde{M}^{st}_\gamma)_\bbQ\to H_{G_\gamma}(\tilde{M}^{st}_\gamma,\bbQ)$ is an isomorphism.
Recall that the variety $\cM _\gamma$ has a cell decomposition \cite{EnRei}. This implies that the cycle map $A(\cM _\gamma)_\bbQ
\to H(\cM _\gamma,\bbQ)$ is an isomorphism, and completes the proof of surjectivity of the map $j^*$. To finish the proof of the proposition it suffices to notice that for each $\gamma \in \bbZ_{\geq 0}^I$ we have $\cH^{\gg 0}_\gamma \subset \ker j^*$ (because $\cM _\gamma$ is a finite dimensional manifold).
\end{proof}

\begin{corollary} The $\cH$-module $\tilde{\cH} ^{st}$ can be considered as a cyclic locally finite $\hat{\cH}$-module via the equivalence
$\cH\-\Mod_{lf}\simeq \hat{\cH}\-\Mod_{lf}$
from Corollary \ref{summary of equivalences}. In particular, the surjection of $\cH$-modules $j^*:\cH \to \tilde{\cH} ^{st}$ induces the surjection of $\hat{\cH}$-modules
$$j^*:\hat{\cH}\to \tilde{\cH}^{st}.$$
\end{corollary}

\subsection{} Similarly, the framed moduli spaces $\{\cM _\gamma\}$ give rise to a graded $\cR$-module. Namely, the equivariant vector bundle $\tau :\tilde{M}_\gamma \to M_\gamma$ induces an isomorphism
$$\cR _\gamma =K_0^{G_\gamma}(M_\gamma )\stackrel{\tau ^*}{\to} K_0^{G_\gamma}(\tilde{M}_\gamma)$$
and diagrams of stacks \eqref{ext of module}, \eqref{ext of module to stable} give rise to a graded $\cR$-module structure on the vector space
$$\cR ^{st}:=\bigoplus _{\gamma \in I} K_0^{G_\gamma}(\tilde{M}^{st}_\gamma)$$
together with a morphism $j^*:\cR \to \cR ^{st}$ of graded $\cR$-modules, coming from the open embedding $j:\tilde{M}^{st}_\gamma \hookrightarrow
\tilde{M}_\gamma$.

\begin{proposition}\label{prop:Rst_locally_finite} (1) The homomorphism of graded $\cR _\bbQ$-modules $j^*:\cR _\bbQ\to \cR ^{st}_\bbQ$ is surjective and its kernel is a large ideal. In particular, $\cR ^{st}_\bbQ $ is a cyclic locally finite $\cR _\bbQ$-module.

(2) Assume that the quiver $Q$ is symmetric. The locally finite $\cH$- and $\cR _\bbQ$-modules $\cH ^{st}$ and $\cR ^{st}_\bbQ $
correspond to each other under the equivalence of categories
$\cH\-\Mod_{lf}\simeq \cR_{\bbQ}\-\Mod_{lf}$
from Corollary \ref{summary of equivalences}.
\end{proposition}

\begin{proof} (1) The (linear) Chern character map $ch:\cR \to \hat{\cH}$ descends to the map $ch:\cR ^{st}\to \cH ^{st}$ in the commutative diagram
$$\begin{array}{ccc}
\cR  & \stackrel{ch}{\to} & \hat{\cH}\\
j^* \downarrow & & \downarrow j^*\\
\cR ^{st} & \stackrel{ch}{\to} & \cH ^{st}
\end{array}
$$

As in the proof of Proposition \ref{prop that geom mod is locally finite}
consider the isomorphisms $\cH ^{st}_\gamma \simeq H(\cM _\gamma ,\bbQ)$ and $\cR ^{st}_\gamma \simeq K^0(\cM _\gamma)$. The map
$$ch:K^0(\cM _\gamma)_\bbQ\to \cH ^{st}_\gamma$$
is an isomorphism, because the space $\cM _\gamma$ has a cell decomposition. Therefore, the map $ch:\cR ^{st}_\bbQ \to \cH ^{st}$
is also an isomorphism. It follows that $j^*:\cR _\bbQ \to \cR ^{st}_\bbQ$ is a surjection of $\cR _\bbQ$-modules. Moreover, for every $\gamma$ and $n$ we have $ch ^{-1}(\hat{\cH}_\gamma ^{\geq n})=\cR _\gamma ^{\geq n}$.
This implies that the kernel of the map $j^*:\cR _\bbQ \to \cR ^{st}_\bbQ$
is a large ideal, i.e. $\cR ^{st}_{\bbQ}$ is a cyclic locally finite $\cR _\bbQ$-module.

(2) Since the map $ch: \cR _\bbQ \to \hat{\cH}$ is not a homomorphism of algebras, the induced map $ch :\cR ^{st}_\bbQ \to \cH ^{st}$ is not a morphism of $\cR_\bbQ$-modules. But we may consider the map $j^*:\hat{\cH}={}^{\tilde{\sigma}}\hat{\cH}\to \cH ^{st}$ as a surjection of ${}^{\tilde{\sigma}}\hat{\cH}$-modules (Proposition \ref{equiv of cat by twists}). Recall that the map
$${}^{\tilde{\sigma}}h:\cR \to {}^{\tilde{\sigma}}\hat{\cH},\quad f\mapsto \eta _\gamma ch(f),\quad \text{for $f\in \cR _\gamma$}$$
is a homomorphism of algebras (see the end of the proof of Theorem \ref{main thm}). Now the key point is the following: for each $\gamma$ the inverse image map $j^*_\gamma:\hat{\cH}_\gamma \to \cH ^{st}_\gamma$ is a ring homomorphism with respect to the natural ring structures on $\hat{\cH}_\gamma $ and $\cH ^{st}_\gamma$. Therefore,
$\eta _\gamma (\ker j^*_\gamma) =\ker j^*_\gamma$.
It follows that the algebra homomorphism ${}^{\tilde{\sigma}}h:\cR \to {}^{\tilde{\sigma}}\hat{\cH}$ descends to the isomorphism of $\cR_\bbQ$-modules
$$\cR ^{st}_\bbQ \to \cH ^{st},\quad m\mapsto \eta _\gamma ch(m),\quad\text{for $m\in \cR _{\gamma ,\bbQ}$}.$$
This completes the proof of the proposition.
\end{proof}

\end{document}